\documentclass{amsart}

\usepackage{amsmath}
\usepackage{amsfonts}
\usepackage{amssymb,enumerate}
\usepackage{amsthm}
\usepackage[all]{xy}
\usepackage{hyperref}

\DeclareMathOperator{\hight}{ht}

\newcommand{\Spec}{\operatorname{Spec}}
\newcommand{\QSpec}{\operatorname{QSpec}}

\newcommand{\GD}{\operatorname{GD}}
\renewcommand{\dim}{\operatorname{dim}}

\newcommand{\Max}{\operatorname{Max}}
\newcommand{\QMax}{\operatorname{QMax}}

\newtheorem{thm}{Theorem}[section]
\newtheorem{cor}[thm]{Corollary}
\newtheorem{lem}[thm]{Lemma}
\newtheorem{prop}[thm]{Proposition}
\newtheorem{defn}[thm]{Definition}

\newtheorem{rem}[thm]{Remark}

\begin{document}

\bibliographystyle{amsplain}

\date{}

\author{Parviz Sahandi}

\address{Department of Mathematics, University of Tabriz, Tabriz, Iran and School of Mathematics, Institute for
Research in Fundamental Sciences (IPM), Tehran Iran.}

\email{sahandi@tabrizu.ac.ir, sahandi@ipm.ir}



\title[Universally catenarian domains and strong S-domains]{Universally catenarian integral domains, strong S-domains and semistar operations}

\begin{abstract} Let $D$ be an integral domain and $\star$ a semistar operation stable and of finite type on it.
In this paper, we are concerned with the study of the semistar
(Krull) dimension theory of polynomial rings over $D$. We
introduce and investigate the notions of $\star$-universally
catenarian and $\star$-stably strong S-domains and prove that,
every $\star$-locally finite dimensional Pr\"{u}fer
$\star$-multiplication domain is $\star$-universally catenarian,
and this implies $\star$-stably strong S-domain. We also give new
characterizations of $\star$-quasi-Pr\"{u}fer domains introduced
recently by Chang and Fontana, in terms of these notions.
\end{abstract}

\maketitle

\section{Introduction}

\noindent The concepts of S(eidenberg)-domains and strong S-domains
are crucial ones and were introduced by Kaplansky \cite[Page 26]{K}.
Recall that an integral domain $D$ is an \emph{S-domain} if for each
prime ideal $P$ of $D$ of height one the extension $PD[X]$ to the
polynomial ring in one variable is also of height one. A
\emph{strong S-domain} is a domain $D$ such that, $D/P$ is an
S-domain, for each prime $P$ of $D$. One of the reasons why
Kaplansky introduced the notion of strong S-domain was to treat the
classes of Noetherian domains and Pr\"{u}fer domains in a unified
frame. Moreover, if $D$ belongs to one of the two classes of
domains, then the following dimension formula holds:
$\dim(D[X_1,\cdots,X_n])=n+\dim(D)$ (cf., \cite[Theorem 9]{S1} and
\cite[Theorem 4]{S2}). The integral domain $D$ is called a
\emph{Jaffard domain} if $\dim(D)<\infty$ and
$\dim(D[X_1,\cdots,X_n])=n+\dim(D)$ for each positive integer $n$.
So that finite dimensional Noetherian or Pr\"{u}fer domains are
Jaffard domains. Kaplansky observed that for $n=1$ and for $D$ a
strong S-domain then $\dim(D[X_1])=1+\dim(D)$ \cite[Theorem 39]{K}.
The strong S-property is not stable, in general under polynomial
extensions (cf. \cite{BMRH}). In \cite{MM}, Malik and Mott, defined
and studied the stably strong S-domains. A domain $D$ is called a
\emph{stably strong S-domain} if $D[X_1,\cdots,X_n]$ is a strong
S-domain for each $n\geq1$. Note that the class of Jaffard domains
contains the class of stably strong S-domains. The class of stably
strong S-domains contains an important class of universally
catenarian domains. Recall that a domain $D$, is called
\emph{catenarian}, if for each pair $P\subset Q$ of prime ideals of
$D$, any two saturated chain of prime ideals between $P$ and $Q$
have the same finite length. If for each $n\geq1$, the polynomial
ring $D[X_1,\cdots,X_n]$ is catenary, then $D$ is said to be
\emph{universally catenarian} (cf. \cite{BDF, BF}).

For several decades, star operations, as described in
\cite[Section 32]{G}, have proven to be an essential tool in
\emph{multiplicative ideal theory}, for studying various classes
of domains. In \cite{OM}, Okabe and Matsuda introduced the concept
of a semistar operation to extend the notion of a star operation.
Since then, semistar operations have been extensively studied
and, because of a greater flexibility than star operations, have
permitted a finer study and new classifications of special
classes of integral domains.

This manuscript is a sequel to \cite{S}. Given a semistar
operation $\star$ on $D$ and let $\widetilde{\star}$ be the
stable semistar operation of finite type canonically associated
to $\star$ (the definitions are recalled later in this section),
it is possible to define a semistar operation stable and of
finite type $\star[X]$ on $D[X]$ (cf. \cite{S}) such that:
$$\widetilde{\star}\text{-}
\dim(D)+1\leq\star[X]\text{-}\dim(D[X])\leq2(\widetilde{\star}\text{-}\dim(D))+1.$$
We say that a domain $D$, is an \emph{$\widetilde{\star}$-Jaffard
domain} if $\widetilde{\star}\text{-}\dim(D)<\infty$ and
$$\star[X_1,\cdots,X_n]\text{-}\dim(D[X_1,\cdots,X_n])=\widetilde{\star}\text{-}\dim(D)+n,$$
for each positive integer $n$. Every $\widetilde{\star}$-Noetherian
and P$\star$MDs are $\widetilde{\star}$-Jaffard domains (cf.
\cite{S}). In this paper we define and study two subclass of
$\widetilde{\star}$-Jaffard domains. Namely in Sections 2 and 3, we
define and study $\widetilde{\star}$-stably strong S-domains and
$\widetilde{\star}$-universally catenarian domains. In Section 4 we
give new characterizations of $\widetilde{\star}$-quasi-Pr\"{u}fer
domains in terms of $\widetilde{\star}$-stably strong S-domains and
$\widetilde{\star}$-universally catenarian domains.

To facilitate the reading of the introduction and of the paper, we
first review some basic facts on semistar operations. Let $D$ denote
a (commutative integral) domain with identity and let $K$ be the
quotient field of $D$. Denote by $\overline{\mathcal{F}}(D)$ the set
of all nonzero $D$-submodules of $K$, and by $\mathcal{F}(D)$ the
set of all nonzero \emph{fractional ideals} of $D$; i.e.,
$E\in\mathcal{F}(D)$ if $E\in\overline{\mathcal{F}}(D)$ and there
exists a nonzero element $r\in D$ with $rE\subseteq D$. Let $f(D)$
be the set of all nonzero finitely generated fractional ideals of
$D$. Obviously,
$f(D)\subseteq\mathcal{F}(D)\subseteq\overline{\mathcal{F}}(D)$. As
in \cite{OM}, a {\it semistar operation on} $D$ is a map
$\star:\overline{\mathcal{F}}(D)\rightarrow\overline{\mathcal{F}}(D)$,
$E\mapsto E^{\star}$, such that, for all $x\in K$, $x\neq 0$, and
for all $E, F\in\overline{\mathcal{F}}(D)$, the following three
properties hold:
\begin{itemize}
\item [$\star_1$]:  $(xE)^{\star}=xE^{\star}$;
\item [$\star_2$]:  $E\subseteq F$ implies that $E^{\star}\subseteq
F^{\star}$;
\item [$\star_3$]:  $E\subseteq E^{\star}$ and
$E^{\star\star}:=(E^{\star})^{\star}=E^{\star}$.
\end{itemize}
It is convenient to say that a \emph{(semi)star operation on} $D$ is
a semistar operation which, when restricted to $\mathcal{F}(D)$, is
a star operation (in the sense of \cite[Section 32]{G}). It is easy
to see that a semistar operation $\star$ on $D$ is a (semi)star
operation on $D$ if and only if $D^{\star}=D$.

Let $\star$ be a semistar operation on the domain $D$. For every
$E\in\overline{\mathcal{F}}(D)$, put $E^{\star_f}:=\bigcup
F^{\star}$, where the union is taken over all finitely generated
$F\in f(D)$ with $F\subseteq E$. It is easy to see that $\star_f$ is
a semistar operation on $D$, and ${\star_f}$ is called \emph{the
semistar operation of finite type associated to} $\star$. Note that
$(\star_f)_f=\star_f$. A semistar operation $\star$ is said to be of
\emph{finite type} if $\star=\star_f$; in particular ${\star_f}$ is
of finite type. We say that a nonzero ideal $I$ of $D$ is a
\emph{quasi-$\star$-ideal} of $D$, if $I^{\star}\cap D=I$; a
\emph{quasi-$\star$-prime} (ideal of $D$), if $I$ is a prime
quasi-$\star$-ideal of $D$; and a \emph{quasi-$\star$-maximal}
(ideal of $D$), if $I$ is maximal in the set of all proper
quasi-$\star$-ideals of $D$. Each quasi-$\star$-maximal ideal is a
prime ideal. It was shown in \cite[Lemma 4.20]{FH} that if
$D^{\star} \neq K$, then each proper quasi-$\star_f$-ideal of $D$ is
contained in a quasi-$\star_f$-maximal ideal of $D$. We denote by
$\QMax^{\star}(D)$ (resp., $\QSpec^{\star}(D)$) the set of all
quasi-$\star$-maximal ideals (resp., quasi-$\star$-prime ideals) of
$D$. When $\star$ is a (semi)star operation, it is easy to see that
the notion of quasi-$\star$-ideal is equivalent to the classical
notion of $\star$-ideal (i.e., a nonzero ideal $I$ of $D$ such that
$I^{\star}=I$).

If $\Delta$ is a set of prime ideals of a domain $D$, then there is
an associated semistar operation on $D$, denoted by
$\star_{\Delta}$, defined as follows:
$$E^{\star_{\Delta}}:=\cap\{ED_P|P\in\Delta\}\text{, for each }E\in\overline{\mathcal{F}}(D).$$
If $\Delta=\emptyset$, let $E^{\star_{\Delta}}:=K$ for each
$E\in\overline{\mathcal{F}}(D)$. One calls $\star_{\Delta}$ the
\emph{spectral semistar operation associated to} $\Delta$. A
semistar operation $\star$ on a domain $D$ is called a
\emph{spectral semistar operation} if there exists a subset
$\Delta$ of the prime spectrum of $D$, Spec$(D)$, such that
$\star=\star_{\Delta}$. When $\Delta:=\QMax^{\star_f}(D)$, we set
$\widetilde{\star}:=\star_{\Delta}$; i.e.,
$$E^{\widetilde{\star}}:= \cap\{ED_P|P\in\QMax^{\star_f}(D)\}\text{, for each
}E\in\overline{\mathcal{F}}(D).$$

It has become standard to say that a semistar operation $\star$ is
{\it stable} if $(E\cap F)^{\star}=E^{\star}\cap F^{\star}$ for all
$E$, $F\in \overline{\mathcal{F}}(D)$. All spectral semistar
operations are stable \cite[Lemma 4.1(3)]{FH}. In particular, for
any semistar operation $\star$, we have that $\widetilde{\star}$ is
a stable semistar operation of finite type \cite[Corollary 3.9]{FH}.

The most widely studied (semi)star operations on $D$ have been
the identity $d_D$, and $v_D$, $t_D:=(v_D)_f$, and
$w_D:=\widetilde{v_D}$ operations, where
$E^{v_D}:=(E^{-1})^{-1}$, with $E^{-1}:=(D:E):=\{x\in
K|xE\subseteq D\}$.

Let $\star$ be a semistar operation on a domain $D$. The
$\star$-{\it Krull dimension of} $D$ is defined as
$$\star\text{-}\dim(D):=\sup\left\{n\bigg| \begin{array}{l} (0)=P_0\subset
P_1\subset\cdots\subset P_n\text{ where } P_i \text{ is a}
\\\text{ quasi-}\star\text{-prime ideal of }D\text{ for } 1\leq i\leq n
 \end{array} \right\}.$$
It is known (see \cite[Lemma 2.11]{EFP}) that
$$
\widetilde{\star}\text{-}\dim(D)=\sup\{\hight (P) \mid P\text{ is a
quasi-}\widetilde{\star}\text{-prime ideal of } D\}.
$$
Thus, if $\star =d_D$, then
$\widetilde{\star}\text{-}\dim(D)=\star\text{-}\dim(D)$ coincides
with $\dim(D)$, the usual (Krull) dimension of $D$.

Let $\star$ be a semistar operation on a domain $D$. Recall from
\cite[Section 3]{EFP} that $D$ is said to be a
\emph{$\star$-Noetherian domain}, if $D$ satisfies the ascending
chain condition on quasi-$\star$-ideals. Also recall from \cite{FJS}
that, $D$ is called a \emph{Pr\"{u}fer $\star$-multiplication
domain} (for short, a P$\star$MD) if each finitely generated ideal
of $D$ is \emph{$\star_f$-invertible}; i.e., if
$(II^{-1})^{\star_f}=D^{\star}$ for all $I\in f(D)$. When $\star=v$,
we recover the classical notion of P$v$MD; when $\star=d_D$, the
identity (semi)star operation, we recover the notion of Pr\"{u}fer
domain.

Let $D$ be a domain, $\star$ a semistar operation on $D$, $T$ an
overring of $D$, and $\iota:D\hookrightarrow T$ the corresponding
inclusion map. In a canonical way, one can define an associated
semistar operation $\star_{\iota}$ on $T$, by setting $E\mapsto
E^{\star_{\iota}}:=E^{\star}$, for each
$E\in\overline{\mathcal{F}}(T)(\subseteq\overline{\mathcal{F}}(D))$
\cite[Proposition 2.8]{FL2}.

Throughout this paper, $D$ denotes a domain and $\star$ is a
semistar operation on $D$.

\section{The $\star$-strong S-domains}

Let $D$ be an integral domain with quotient field $K$, let $X$, $Y$
be two indeterminates over $D$ and let $\star$ be a semistar
operation on D. Set $D_1:=D[X]$, $K_1:=K(X)$ and take the following
subset of $\Spec(D_1)$:
$$\Theta_1^{\star}:=\{Q_1\in\Spec(D_1)|\text{ }Q_1\cap D=(0)\text{ or }(Q_1\cap D)^{\star_f}\subsetneq D^{\star}\}.$$
Set
$\mathfrak{S}_1^{\star}:=\mathcal{S}(\Theta_1^{\star}):=D_1[Y]\backslash(\bigcup\{Q_1[Y]
|Q_1\in\Theta_1^{\star}\})$ and:
$$E^{\circlearrowleft_{\mathfrak{S}_1^{\star}}}:=E[Y]_{\mathfrak{S}_1^{\star}}\cap
K_1, \text{   for all }E\in \overline{\mathcal{F}}(D_1).$$

It is proved in \cite[Theorem 2.1]{S} that the mapping
$\star[X]:=\circlearrowleft_{\mathfrak{S}_1^{\star}}:
\overline{\mathcal{F}}(D_1)\to\overline{\mathcal{F}}(D_1)$,
$E\mapsto E^{\star[X]}$ is a
stable semistar operation of finite type on $D[X]$, i.e.,
$\widetilde{\star[X]}=\star[X]$. It is also proved that
$\widetilde{\star}[X]=\star_f[X]=\star[X]$, $d_D[X]=d_{D[X]}$ and
$\QSpec^{\star[X]}(D[X])=\Theta_1^{\star}\backslash\{0\}$. If
$X_1,\cdots,X_r$ are indeterminates over $D$, for $r\geq2$, we let
$$\star[X_1,\cdots,X_r]:=(\star[X_1,\cdots,X_{r-1}])[X_r],$$ where
$\star[X_1,\cdots,X_{r-1}]$ is a stable semistar operation of finite
type on $D[X_1,\cdots,X_{r-1}]$. For an integer $r$, put $\star[r]$
to denote $\star[X_1,\cdots,X_r]$ and $D[r]$ to denote
$D[X_1,\cdots,X_r]$.

As an extension of a result by Seidenberg \cite[Theorem 2]{S1}, we
showed in \cite[Theorem 3.1]{S} that: if
$n:=\widetilde{\star}$-$\dim(D)$, then
$n+1\leq\star[X]\text{-}\dim(D[X])\leq 2n+1.$ On the other hand, it
is shown in \cite[Theorem 3.8 and Corollary 4.11]{S}, that if $D$ is
a $\widetilde{\star}$-Noetherian domain or a P$\star$MD and $n$ is
any positive integer, then
$\star[n]\text{-}\dim(D[n])=n+\widetilde{\star}\text{-}\dim(D)$,
that is $D$ is an $\widetilde{\star}$-Jaffard domain. Now we define
and study a subclass of $\widetilde{\star}$-Jaffard domains.

\begin{defn}
The domain $D$ is called an \emph{$\star$-S-domain}, if each height
one quasi-$\star$-prime ideal $P$ of $D$, extends to a height one
quasi-$\star[X]$-prime ideal $P[X]$ of the polynomial ring $D[X]$.
We say that $D$ is an \emph{$\star$-strong S-domain}, if each pair
of adjacent quasi-$\star$-prime ideals $P_1\subset P_2$ of $D$,
extend to a pair of adjacent quasi-$\star[X]$-prime ideals
$P_1[X]\subset P_2[X]$, of $D[X]$. If for each $n\geq1$, the
polynomial ring $D[n]$ is a $\star[n]$-strong S-domain, then $D$ is
said to be an \emph{$\star$-stably strong S-domain}.
\end{defn}

Note that the notion of $d$-S-domain (resp. $d$-strong S-domain,
$d$-stably strong S-domain) coincides with the ``classical'' notion
of S-domain (resp. strong S-domain, stably strong S-domain) \cite{K,
MM}.

\begin{prop} Suppose that $D$ is an $\widetilde{\star}$-strong S-domain,
and $\widetilde{\star}$-$\dim(D)=n$ is finite. Then
$\star[X]$-$\dim(D[X])=n+1$.
\end{prop}

\begin{proof} By \cite[Theorem 3.1]{S}, we only have to show that $\star[X]$-$\dim(D[X])\leq
n+1$. Let $Q\in\QSpec^{\star[X]}(D[X])$ and set $P:=Q\cap D$.
There are two cases to consider. If $P=0$, then by \cite[Theorem
37]{K}, $\hight(Q)\leq1$. If $P\neq0$, we show that
$P\in\QSpec^{\widetilde{\star}}(D)$. Let
$M\in\QMax^{\star[X]}(D[X])$ containing $Q$ \cite[Lemma 2.3
(1)]{FL}. Since $0\neq P\subseteq M\cap D$, then $M\cap
D\in\QSpec^{\star_f}(D)$ by \cite[Remark 2.3]{S}. Hence
$P\in\QSpec^{\widetilde{\star}}(D)$ by \cite[Lemma 4.1, Remark
4.5]{FH}. Consequently $\hight(P)\leq n$ by the hypothesis. By an
argument the same as \cite[Theorem 39]{K}, we obtain that
$\hight(P)=\hight(P[X])$. Since $P[X]\subseteq Q$, then we have
$\hight(Q)\leq n+1$. Therefore $\star[X]$-$\dim(D[X])\leq n+1$ as
desired.
\end{proof}

\begin{cor} Each $\widetilde{\star}$-stably strong S-domain of finite $\widetilde{\star}$-dimension is an $\widetilde{\star}$-Jaffard domain.
\end{cor}

\begin{prop}\label{SSD} Let $D$ be an integral domain. The following then are equivalent:
\begin{itemize}
\item[(1)] $D$ is an $\widetilde{\star}$-S-domain (resp. $\widetilde{\star}$-strong
S-domain).

\item[(2)] $D_P$ is an S-domain (resp. strong S-domain) for all $P\in\QSpec^{\widetilde{\star}}(D)$.

\item[(3)] $D_M$ is an S-domain (resp. strong S-domain) for all $M\in\QMax^{\widetilde{\star}}(D)$.
\end{itemize}
\end{prop}

\begin{proof} For either cases follow the
method of \cite[Proposition 2.1]{MM}, and note that every
quasi-$\widetilde{\star}$-prime ideal of $D$ in contained in a
quasi-$\widetilde{\star}$-maximal ideal by \cite[Lemma 2.3 (1)]{FL}.
\end{proof}

\begin{prop}\label{SSSD} Let $D$ be an integral domain. The following then are equivalent:
\begin{itemize}
\item[(1)] $D$ is an $\widetilde{\star}$-stably strong S-domain.

\item[(2)] $D_P$ is an stably strong S-domain, for all $P\in\QSpec^{\widetilde{\star}}(D)$.

\item[(3)] $D_M$ is an stably strong S-domain, for all $M\in\QMax^{\widetilde{\star}}(D)$.
\end{itemize}
\end{prop}

\begin{proof} $(1)\Rightarrow(2)$. Suppose that $D$ is an $\widetilde{\star}$-stably strong
S-domain, $P\in\QSpec^{\widetilde{\star}}(D)$ and $n\geq1$ is an
integer. It suffices by \cite[Lemma 6.3.1]{FHP}, to show that for
each maximal ideal $\mathcal{M}$ of $D_P[n]$, the local ring
$D_P[n]_{\mathcal{M}}$ is a strong S-domain. To this end, let
$\mathcal{M}$ be an arbitrary maximal ideal of $D_P[n]$. Note that
$D_P[n]=D[n]_{D\backslash P}$. So that there exists a prime ideal
$M$ of $D[n]$ such that $M\cap (D\backslash P)=\emptyset$ and
$\mathcal{M}=MD_P[n]$. Consequently $M\cap D\subseteq P$, and
therefore by \cite[Remark 2.3]{S}, we have
$M\in\QSpec^{\star[n]}(D[n])$. Now since by the hypothesis $D[n]$ is
an $\star[n]$-strong S-domain, we then have
$D_P[n]_{\mathcal{M}}=D[n]_M$ is a strong S-domain domain by
Proposition \ref{SSD}.

$(2)\Rightarrow(1)$. Let $n\geq1$ be an integer,
$Q\in\QSpec^{\star[n]}(D[n])$ and set $P:=Q\cap D$. We plan to show
that $D[n]_Q$ is an strong S-domain domain. If $P=0$, then
$D[n]_Q=K[n]_{QK[n]}$, which is an strong S-domain domain since it
is Noetherian and \cite[Theorem 149]{K}. If $P\neq0$, then
$P\in\QSpec^{\widetilde{\star}}(D)$. Therefore
$D[n]_Q=D_P[n]_{QD_P[n]}$, and hence is an strong S-domain by
hypothesis. So that $D[n]$ is an $\star[n]$-strong S-domain by
Proposition \ref{SSD}. Thus $D$ is an $\widetilde{\star}$-stably
strong S-domain by the definition.

$(2)\Leftrightarrow(3)$ is true, using \cite[Lemma 6.3.1]{FHP}.
\end{proof}

Recall from \cite{CF} that $D$ is said to be a
\emph{$\star$-quasi-Pr\"{u}fer domain}, in case, if $Q$ is a prime
ideal in $D[X]$, and $Q\subseteq P[X]$, for some
$P\in\QSpec^{\star}(D)$, then $Q=(Q\cap D)[X]$. This notion is the
semistar analogue of the classical notion of the quasi-Pr\"{u}fer
domains \cite[Section 6.5]{FHP} (that is among other equivalent
conditions, the domain $D$ is said to be a \emph{quasi-Pr\"{u}fer
domain} if it has Pr\"{u}ferian integral closure). By
\cite[Corollary 2.4]{CF}, $D$ is a $\star_f$-quasi-Pr\"{u}fer
domain if and only if $D$ is a
$\widetilde{\star}$-quasi-Pr\"{u}fer domain.

\begin{cor} If $D$ is a $\widetilde{\star}$-Noetherian (resp. a $\widetilde{\star}$-quasi-Pr\"{u}fer) domain,
then $D$ is an $\widetilde{\star}$-stably strong S-domain.
\end{cor}

\begin{proof} Let $P\in\QSpec^{\widetilde{\star}}(D)$. Since $D_P$ is a
Noetherian domain by \cite[Proposition 3.8]{EFP} (resp. a
quasi-Pr\"{u}fer domain by \cite[Lemma 2.1]{CF}), we obtain that it
is an stably strong S-domain by \cite[Theorem 149]{K} (resp. by
\cite[Corollary 6.7.6]{FHP}). Therefore $D$ is an
$\widetilde{\star}$-stably strong S-domain by Proposition
\ref{SSSD}.
\end{proof}

Therefore we have the following implications for finite
$\widetilde{\star}$-dimensional domains:
$$ \textsf{$\widetilde{\star}$-Noetherian}\text{ or }
\textsf{$\widetilde{\star}$-quasi-Pr\"{u}fer}\Rightarrow
\textsf{$\widetilde{\star}$-stably strong S-domain}\Rightarrow
\textsf{$\widetilde{\star}$-Jaffard}.$$

\begin{rem} Call a domain $D$, an \emph{$\widetilde{\star}$-locally Jaffard domain} if
$D_P$ is a Jaffard domain for each
$P\in\QSpec^{\widetilde{\star}}(D)$. Therefore every
$\widetilde{\star}$-stably strong S-domain is a
$\widetilde{\star}$-locally Jaffard domain. It is not hard to prove
that every $\widetilde{\star}$-locally Jaffard domain is an
$\widetilde{\star}$-Jaffard domain (see proof of \cite[Theorem
3.2]{S}).
\end{rem}

\section{The $\star$-catenarian domains}

In this section we introduce and study a subclass of
$\widetilde{\star}$-stably strong S-domains, namely
$\widetilde{\star}$-universally catenarian domains.

\begin{defn} The domain $D$ is called \emph{$\star$-catenary}, if for each pair
$P\subset Q$ of quasi-$\star$-prime ideals  of $D$, any two
saturated chain of quasi-$\star$-prime ideals between $P$ and $Q$
have the same finite length. If for each $n\geq1$, the polynomial
ring $D[n]$ is $\star[n]$-catenary, then $D$ is said to be
\emph{$\star$-universally catenarian}.
\end{defn}

Note that the notion of $d$-catenary (resp. $d$-universally
catenarian) coincides with the ``classical'' notion of catenary
(resp. universally catenarian). The proof of the the following
proposition is straightforward, so we omit it.

\begin{prop}\label{caloc} Let $D$ be an integral domain. The following then are equivalent:
\begin{itemize}
\item[(1)] $D$ is $\widetilde{\star}$-catenary.

\item[(2)] $D_P$ is catenary for all $P\in\QSpec^{\widetilde{\star}}(D)$.

\item[(3)] $D_M$ is catenary for all $M\in\QMax^{\widetilde{\star}}(D)$.
\end{itemize}
\end{prop}

\begin{lem}\label{loc} Let $D$ be an integral domain and $n\geq1$ be an integer. Then $D[n]$ is
$\star[n]$-catenary, if and only if, $D_P[n]$ is catenary for all
$P\in\QSpec^{\widetilde{\star}}(D)$.
\end{lem}

\begin{proof} $(\Rightarrow)$ It suffices by \cite[Lemma 6.3.2]{FHP},
to show that for each maximal ideal $\mathcal{M}$ of $D_P[n]$, the
local ring $D_P[n]_{\mathcal{M}}$ is catenary. To this end, let
$\mathcal{M}$ be an arbitrary maximal ideal of $D_P[n]$. Note that
$D_P[n]=D[n]_{D\backslash P}$. So that there exists a prime ideal
$M$ of $D[n]$ such that $M\cap (D\backslash P)=\emptyset$ and
$\mathcal{M}=MD_P[n]$. Consequently $M\cap D\subseteq P$, and
therefore by \cite[Remark 2.3]{S}, we have
$M\in\QSpec^{\star[n]}(D[n])$. Hence $D_P[n]_{\mathcal{M}}=D[n]_M$
is a catenary domain.

$(\Leftarrow)$ Let $Q\in\Max^{\star[n]}(D[n])$ and set $P:=Q\cap D$.
We plan to show that $D[n]_Q$ is a catenarian domain. There are two
cases to consider. If $P=0$, then $D[n]_Q=K[n]_{QK[n]}$, which is a
catenarian domain since it is a Cohen-Macaulay ring and
\cite[Theorem 2.1.12]{BH}. If $P\neq0$, then
$P\in\QSpec^{\widetilde{\star}}(D)$. Therefore
$D[n]_Q=D_P[n]_{QD_P[n]}$, which is catenary by the hypothesis.
Whence $D[n]$ is $\star[n]$-catenary by Proposition \ref{caloc}.
\end{proof}

It is convenient to say that, a domain $D$ is
\emph{$\widetilde{\star}$-locally finite dimensional} (for short,
$\widetilde{\star}$-LFD) if $\hight(P)<\infty$ for every
$P\in\QSpec^{\widetilde{\star}}(D)$. The special case $\star=d_D$ of
the following theorem is contained in \cite[Theorem 12]{BF}.

\begin{thm} If $D$ is a P$\star$MD which is $\widetilde{\star}$-LFD, then $D$ is $\widetilde{\star}$-universally catenarian.
\end{thm}

\begin{proof} We have to show that for each integer $n\geq1$, $D[n]$ is a $\star[n]$-catenarian domain.
To this end let $n\geq1$ be an integer and
$P\in\QSpec^{\widetilde{\star}}(D)$. So that $D_P$ is a finite
dimensional valuation domain by the hypothesis and \cite[Theorem
3.1]{FJS}. Hence $D_P[n]$ is catenary by \cite[Theorem 12]{BF}. Thus
$D[n]$ is $\star[n]$-catenary by Lemma \ref{loc}, for all $n\geq1$.
Hence $D$ is $\widetilde{\star}$-universally catenarian.
\end{proof}

\begin{prop} Let $D$ be an integral domain. If $D[X]$ is $\star[X]$-catenarian, then $D$ is
an $\widetilde{\star}$-strong S-domain.
\end{prop}

\begin{proof} Using Lemma \ref{loc}, $D_P[X]$ is a catenarian domain for all
$P\in\QSpec^{\widetilde{\star}}(D)$. Hence $D_P$ is a strong
S-domain by \cite[Lemma 2.3]{BDF}. Thus $D$ is a
$\widetilde{\star}$-strong S-domain by Proposition \ref{SSD}.
\end{proof}

\begin{cor} Each $\widetilde{\star}$-universally catenarian domain
is an $\widetilde{\star}$-stably strong S-domain.
\end{cor}

Therefore we have the following implications for finite
$\widetilde{\star}$-dimensional domains:
$$ \textsf{P$\star$MD}\Rightarrow\textsf{$\widetilde{\star}$-universally
catenary}\Rightarrow \textsf{$\widetilde{\star}$-stably strong
S-domain}\Rightarrow \textsf{$\widetilde{\star}$-Jaffard}.$$

Next we wish to present the semistar analogue of the celebrated
theorem of Ratliff \cite[Theorem 2.6]{Rat}.

\begin{thm} Let $D$ be an integral domain. Suppose that $D$ is $\widetilde{\star}$-Noetherian.
Then $D[X]$ is $\star[X]$-catenary if and only if $D$ is
$\widetilde{\star}$-universally catenarian.
\end{thm}

\begin{proof} The ``if'' part is trivial. For the ``only if'' part, let $n\geq1$ be an integer and
$P\in\QSpec^{\widetilde{\star}}(D)$. So by Lemma \ref{loc}, $D_P[X]$
is catenary. Since $D_P$ is a Noetherian domain by \cite[Proposition
3.8]{EFP}, using the result of Ratliff \cite[Theorem 2.6]{Rat}, we
have $D_P$ is a universally catenarian domain. Therefore $D_P[n]$ is
a catenarian domain. Thus another use of Lemma \ref{loc} yields us
that $D[n]$ is a $\star[n]$-catenarian domain for all $n\geq1$.
Hence $D$ is an $\widetilde{\star}$-universally catenarian.
\end{proof}

\section{Characterizations of $\star$-quasi-Pr\"{u}fer domains}

In this section we give some characterization of
$\widetilde{\star}$-quasi-Pr\"{u}fer domains. First we need to
recall the definition of a semistar going-down domain. Let
$D\subseteq T$ be an extension of domains. Let $\star$ and
$\star'$ be semistar operations on $D$ and $T$, respectively.
Following \cite{DS}, we say that $D\subseteq T$ satisfies
$(\star,\star')$-$\GD$ if, whenever $P_0\subset P$ are
quasi-$\star$-prime ideals of $D$ and $Q$ is a
quasi-$\star'$-prime ideal of $T$ such that $Q\cap D=P$, there
exists a quasi-$\star'$-prime ideal $Q_0$ of $T$ such that
$Q_0\subseteq Q$ and $Q_0\cap D= P_0$. The integral domain $D$ is
said to be a \emph{$\star$-going-down domain} (for short, a
$\star$-$\GD$ {\it domain}) if, for every overring $T$ of $D$ and
every semistar operation $\star'$ on $T$, the extension
$D\subseteq T$ satisfies $(\star,\widetilde{\star'})$-$\GD$.
These concepts are the semistar versions of the ``classical''
concepts of going-down property and the going-down domains (cf.
\cite{DP}). It is known by \cite[Propositions 3.5 and 3.2(e)]{DS}
that every P$\star$MD and every integral domain $D$ with
$\star$-$\dim(D)=1$ is a $\star$-$\GD$ domain.

\begin{thm}\label{qJafgd} Let $D$ be an integral domain. Suppose that $D$ is a $\widetilde{\star}$-$\GD$ domain which is
$\widetilde{\star}$-$LFD$. Then the following statements are
equivalent:
\begin{itemize}
\item[(1)] $D$ is an $\widetilde{\star}$-universally catenarian domain.

\item[(2)] $D$ is an $\widetilde{\star}$-stably strong S-domain.

\item[(3)] $D$ is an $\widetilde{\star}$-strong S-domain.

\item[(4)] $D$ is an $\widetilde{\star}$-locally Jaffard domain.

\item[(5)] $D_M$ is a Jaffard domain for each $M\in\QMax^{\widetilde{\star}}(D)$.

\item[(6)] $D[X]$ is an $\star[X]$-catenarian domain.

\item[(7)] $D$ is an $\widetilde{\star}$-quasi-Pr\"{u}fer domain.
\end{itemize}
\end{thm}

\begin{proof} First of all we show that for each $P\in\QSpec^{\widetilde{\star}}(D)$, $D_P$ is a going-down domain.
Let $T$ be an overring of $D_P$. Suppose that $P_1D_P\subset P_2D_P$
are prime ideals of $D_P$ and $Q_2$ is a prime ideal of $T$ such
that $Q_2\cap D_P=P_2D_P$. Since $P_1\subset P_2$ are
quasi-$\widetilde{\star}$-prime ideals of $D$ (since they are
contained in $P$ and \cite[Lemma 4.1 and Remark 4.5]{FH}) and
$Q_2\cap D=P_2$ and the fact that $D$ is a $\widetilde{\star}$-$\GD$
domain, there exists a (quasi-$\widetilde{\star}$-)prime ideal $Q_2$
of $T$ satisfying both $Q_1\subseteq Q_2$ and $Q_1\cap D=P_1$. So
that $Q_1\cap D_P=P_1D_P$. Therefore $D_P$ is a going-down domain.

The implications $(1)\Rightarrow(2)$, $(2)\Rightarrow(3)$, and
$(4)\Rightarrow(5)$ are already known (see Section 3).

$(3)\Rightarrow(4)$. Let $P\in\QSpec^{\widetilde{\star}}(D)$.
Therefore $D_P$ is a going-down domain which is also a strong
S-domain by Proposition \ref{SSD}. Hence by \cite[Theorem
1.13]{ABDFK}, $D_P$ is a Jaffard domain. Thus $D$ is an
$\widetilde{\star}$-locally Jaffard.

$(5)\Rightarrow(6)$. Let $P\in\QSpec^{\widetilde{\star}}(D)$. Choose
a quasi-$\widetilde{\star}$-maximal ideal $M$ of $D$ containing $P$.
Since $D_M$ is a Jaffard domain which is also going-down, then
\cite[Theorem 1.13]{ABDFK} tells us that $D_M[X]$ is catenarian.
Thus $D_P[X]=(D_M[X])_{D\backslash P}$ is a catenarian domain. Now
by Lemma \ref{loc}, $D[X]$ is $\star[X]$-catenarian.

$(6)\Rightarrow(7)$. Let $P\in\QSpec^{\widetilde{\star}}(D)$. Using
Lemma \ref{loc}, $D_P[X]$ is catenarian. Since $D_P$ is a going-down
domain, using \cite[Theorem 1.13]{ABDFK}, we see that $D_P$ is a
quasi-Pr\"{u}fer domain. Thus by \cite[Lemma 2.1]{CF}, $D$ is a
$\widetilde{\star}$-quasi-Pr\"{u}fer domain.

$(7)\Rightarrow(1)$. Let $P\in\QSpec^{\widetilde{\star}}(D)$. Since
$D_P$ is a quasi-Pr\"{u}fer going-down domain, we have $D_P$ is a
universally catenarian domain by \cite[Theorem 1.13]{ABDFK}. Hence
$D_P[n]$ is a catenarian domain for each integer $n\geq1$. Therefore
using Lemma \ref{loc} we obtain that $D[n]$ is a
$\star[n]$-catenarian domain for all $n\geq1$. Hence $D$ is an
$\widetilde{\star}$-universally catenarian domain.
\end{proof}

\begin{cor} Let $D$ be an integral domain. Suppose that $\widetilde{\star}$-$\dim(D)=1$. Then the
following statements are equivalent:
\begin{itemize}
\item[(1)] $D$ is an $\widetilde{\star}$-universally catenarian domain.

\item[(2)] $D[X]$ is an $\star[X]$-catenarian domain.

\item[(3)] $D$ is an $\widetilde{\star}$-stably strong S-domain.

\item[(4)] $D$ is an $\widetilde{\star}$-strong S-domain.

\item[(5)] $D$ is an $\widetilde{\star}$-S-domain.

\item[(6)] $\star[X]$-$\dim(D[X])=2$.

\item[(7)] $D$ is an $\widetilde{\star}$-Jaffard domain.

\item[(8)] $D$ is an $\widetilde{\star}$-locally Jaffard domain.

\item[(9)] $D$ is an $\widetilde{\star}$-quasi-Pr\"{u}fer domain.
\end{itemize}
\end{cor}

\begin{proof} Observe that
$(1)\Leftrightarrow(2)\Leftrightarrow(3)\Leftrightarrow(4)\Leftrightarrow(8)\Leftrightarrow(9)$
by Theorem \ref{qJafgd}, since $D$ is a $\widetilde{\star}$-$\GD$
domain and $(4)\Rightarrow(5)$ is trivial.

$(5)\Rightarrow(6)$. By \cite[Theorem 3.1]{S}, we have
$2\leq\star[X]$-$\dim(D[X])$. Now let $P$ be a
quasi-$\widetilde{\star}$-prime ideal of $D$. Hence by the
hypothesis, $\hight(P)=1$. Since $D$ is a
$\widetilde{\star}$-S-domain, $\hight(P[X])=1$. Now let $Q$ be a
quasi-$\star[X]$-prime ideal of $D[X]$. Take a prime ideal $Q_1$
properly contained in $Q$. If it contracts to a
$P\in\QSpec^{\widetilde{\star}}(D)$, then by \cite[Theorem 37]{K},
$Q_1=P[X]$. If it contracts to zero, then by \cite[Theorem 37]{K},
we have $\hight(Q_1)\leq1$. Whence in either cases we have
$\hight(Q)\leq2$. Therefore $\star[X]$-$\dim(D[X])\leq2$. Thus
$\star[X]$-$\dim(D[X])=2$.

$(6)\Leftrightarrow(7)$ is true by \cite[Corollary 4.12]{S} and
$(6)\Rightarrow(9)$ by \cite[Theorem 3.5]{S}.
\end{proof}

Next we characterize $\widetilde{\star}$-quasi-Pr\"{u}fer domains by
the means of the properties of their overrings. Before that we need
to recall the definition of $(\star,\star')$-linked overrings. Let
$D$ be a domain and $T$ an overring of $D$. Let $\star$ and $\star'$
be semistar operations on $D$ and $T$, respectively. One says that
$T$ is \emph{$(\star,\star')$-linked to} $D$ (or that $T$ is a
$(\star,\star')${\it -linked overring of} $D$) if
$F^{\star}=D^{\star}\Rightarrow (FT)^{\star'}=T^{\star'}$, when $F$
is a nonzero finitely generated ideal of $D$ (cf. \cite{EF}).

Let $T$ be a $(\star,\star')$-linked overring of $D$. We will show
that the contraction map on prime spectra restricts to a well
defined function
$$
G:\QSpec^{\widetilde{\star'}}(T)\to \QSpec^{\widetilde{\star}}(D),
\, \, Q \mapsto Q \cap D
$$
of topological spaces which is continuous (with respect to the
subspace topology induced by the Zariski topology). If
$Q\in\QSpec^{\widetilde{\star'}}(T)$, then we show that $P:=G(Q)=Q
\cap D$ is a quasi-$\widetilde{\star}$-ideal of $D$. To this end
it suffices to show that $P^{\widetilde{\star}}\neq
D^{\widetilde{\star}}$. So suppose that
$P^{\widetilde{\star}}=D^{\widetilde{\star}}$. Then
$(PT)^{\widetilde{\star'}}=T^{\widetilde{\star'}}$. Since
$PT\subseteq Q$ we obtain that
$Q^{\widetilde{\star'}}=T^{\widetilde{\star'}}$, and hence $Q=T$,
which is a contradiction.

In the following theorem, let $\star'$ be a semistar operation for
an overring $T$ of $D$.

\begin{thm}\label{qJaf} Let $D$ be an integral domain. Suppose that $\widetilde{\star}$-$\dim(D)$ is finite. Then the
following statements are equivalent:
\begin{itemize}
\item[(1)] Each $(\star,\star')$-linked overring $T$ of $D$ is an
$\widetilde{\star'}$-universally catenarian domain.

\item[(2)] Each $(\star,\star')$-linked overring $T$ of $D$ is an
$\widetilde{\star'}$-stably strong S-domain.

\item[(3)] Each $(\star,\star')$-linked overring $T$ of $D$ is an
$\widetilde{\star'}$-strong S-domain.

\item[(4)] Each $(\star,\star')$-linked overring $T$ of $D$ is an
$\widetilde{\star'}$-Jaffard domain.

\item[(5)] Each $(\star,\star')$-linked overring $T$ of $D$ is an
$\widetilde{\star'}$-quasi-Pr\"{u}fer domain.

\item[(6)] $D$ is an $\widetilde{\star}$-quasi-Pr\"{u}fer domain.
\end{itemize}
\end{thm}

\begin{proof} Note that $(1)\Rightarrow(2)$ and
$(2)\Rightarrow(3)$ are trivial.

$(3)\Rightarrow(6)$. Let $M$ be a quasi-$\star_f$-maximal ideal of
$D$. We wish to show that $D_M$ is a quasi-Pr\"{u}fer domain.
Suppose that $T$ is an overring of $D_M$. Since $T$ is a
$(\star,\star_{\iota})$-linked overring of $D$, we have $T$ is a
$\widetilde{\star}_{\iota}$-strong S-domain by the hypothesis, where
$\iota$ is the canonical inclusion of $D$ into $T$. We want to show
that $\QSpec^{\widetilde{\star}_{\iota}}(T)\cup\{0\}=\Spec(T)$. So
let $Q$ be an arbitrary non-zero prime ideal of $T$, and set
$PD_M:=Q\cap D_M$, where $P\in\Spec(D)$ such that $P\subseteq M$.
Note that $P$ is a quasi-$\widetilde{\star}$-prime ideal of $D$,
since it is contained in $M$ and \cite[Lemma 4.1 and Remark
4.5]{FH}, and that $P=Q\cap D$. If
$Q^{\widetilde{\star}_{\iota}}=T^{\widetilde{\star}_{\iota}}$, that
is, if $Q^{\widetilde{\star}}=T^{\widetilde{\star}}$, then we have
$Q^{\widetilde{\star}}\cap D=D$. But
$$P=P^{\widetilde{\star}}\cap D=(Q\cap D)^{\widetilde{\star}}\cap
D=Q^{\widetilde{\star}}\cap D=D,$$ which is a contradiction.
Therefore $Q^{\widetilde{\star}_{\iota}}\neq
T^{\widetilde{\star}_{\iota}}$, and hence
$Q\in\QSpec^{\widetilde{\star}_{\iota}}(T)$ since
$\widetilde{\star}_{\iota}=\widetilde{(\widetilde{\star}_{\iota})}$
is a stable semistar operation of finite type, and so
$\QSpec^{\widetilde{\star}_{\iota}}(T)\cup\{0\}=\Spec(T)$. This
means that $T$ is a strong S-domain. Therefore thanks to
\cite[Theorem 6.7.8]{FHP}, $D_M$ is a quasi-Pr\"{u}fer domain. Hence
$D$ is a $\widetilde{\star}$-quasi-Pr\"{u}fer domain by \cite[Lemma
2.1]{CF}.

$(4)\Leftrightarrow(5)\Leftrightarrow(6)$ was proved in
\cite[Theorem 4.14]{S}.

$(6)\Rightarrow(1)$. Suppose that $T$ is a $(\star,\star')$-linked
overring of $D$. Let $n\geq1$ be an integer and
$Q\in\QSpec^{\star'[n]}(T[n])$. Set $Q_0=Q\cap T$ which is a
quasi-$\widetilde{\star'}$-prime ideal of $T$ (or equal to zero).
Then $P:=Q_0\cap D$ is a quasi-$\widetilde{\star}$-prime ideal of
$D$ by the observation before the theorem (or equal to zero). Thus
$D_P$ is a quasi-Pr\"{u}fer domain by \cite[Lemma 2.1]{CF}. Since
$D_P\subseteq T_{Q_0}$, we find that $T_{Q_0}$ is a universally
catenarian domain by \cite{AC}. Thus
$T[n]_Q=T_{Q_0}[n]_{QT_{Q_0}[n]}$ is a catenarian domain.
Consequently $T[n]$ is $\star'[n]$-catenary for each integer
$n\geq1$, that is $T$ is an $\widetilde{\star'}$-universally
catenarian domain.
\end{proof}

\begin{center} {\bf ACKNOWLEDGMENT}

\end{center}

I would like to thank Professor Marco Fontana for his comments on
this paper. I also thank the referee for several helpful remarks
concerning the final form of the paper.

\end{document}